\documentclass{article}
\usepackage{amsmath,amsfonts,amssymb}

\usepackage{mathtools}
\usepackage{amsthm}
\usepackage[pagewise]{lineno}

\newtheorem{theorem}{Theorem}
\newtheorem{lemma}{Lemma}

\title{Solubility of Additive Sextic Forms over $\mathbb{Q}_2(\sqrt{-1})$ and $\mathbb{Q}_2(\sqrt{-5})$}

\author{Drew Duncan \and David B. Leep}

\begin{document}

\maketitle

\begin{abstract}

Michael Knapp, in a previous work, conjectured that every additive sextic form over $\mathbb{Q}_2(\sqrt{-1})$ and $\mathbb{Q}_2(\sqrt{-5})$ in seven variables has a nontrivial zero. In this paper, we show that this conjecture is true, establishing that $\Gamma^*(6, \mathbb{Q}_2(\sqrt{-1})) = \Gamma^*(6, \mathbb{Q}_2(\sqrt{-5})) = 7 $.
\end{abstract}

\section{Introduction}
For equations of the form
\begin{equation}
\label{one}
a_1x_1^d + a_2x_2^d + \ldots + a_sx_s^d = 0
\end{equation}
with $a_i$ belonging to some $\pi$-adic field $K$ (an algebraic extension of $\mathbb{Q}_p$, the field of $p$-adic numbers), let $\Gamma^*(d, K)$ represent the minimum number of variables $s$ such that equation (\ref{one}) with exponent $d$ is guaranteed to have a nontrivial solution, regardless of the choice of the coefficients $a_i$.

In \cite{davenport1963homogeneous}, Davenport and Lewis defined an operation of \textit{contraction} which can be applied to such equations, and used this to demonstrate that $\Gamma^*(d, \mathbb{Q}_p) \le d^2+1$ for every prime $p$.  Furthermore, they demonstrated that equality holds whenever $d+1 = p$.  However, it is often the case that a much smaller number of variables suffices to guarantee the existence of a nontrivial solution, and so one may ask for an exact value for $\Gamma^*(d, K).$

Knapp \cite{knapp2016solubility}, using the method of contraction, showed that for all six ramified quadratic extensions $K$ of $\mathbb{Q}_2$, we have $\Gamma^*(6,K) \le 9$, with equality holding for $\mathbb{Q}_2(\sqrt{2})$, $\mathbb{Q}_2(\sqrt{-2})$, $\mathbb{Q}_2(\sqrt{10})$, and $\mathbb{Q}_2(\sqrt{-10})$.  For the remaining two extensions ($\mathbb{Q}_2(\sqrt{-1})$ and $\mathbb{Q}_2(\sqrt{-5})$), Knapp further showed that $\Gamma^*(6,K) \ge 7$ and conjectured that $\Gamma^*(6,K) = 7$.  Indeed, it is straightforward to see that $x_0^6 + \pi x_1^6 + \ldots + \pi^5 x_5^6 = 0$ has no nontrivial solution over the fields considered.  In this paper, we show that this conjecture is correct.

\begin{theorem}
We have $$\Gamma^*(6, \mathbb{Q}_2(\sqrt{-1})) = \Gamma^*(6, \mathbb{Q}_2(\sqrt{-5})) = 7. $$
\end{theorem}

\section{Preliminary Concepts}

Let $K$ denote one of the ramified quadratic extensions $\mathbb{Q}_2(\sqrt{-1})$ or $\mathbb{Q}_2(\sqrt{-5})$, and $\mathcal{O}$ denote its ring of integers.  Without loss of generality, assume $a_i \in \mathcal{O}$.  For the remainder, assume that the number of variables $s$ in equation (\ref{one}) is 7 and the exponent $d$ is 6, i.e,
\begin{equation}
\label{eq}
a_1x_1^6 + a_2x_2^6 + \ldots + a_7x_7^6 = 0.
\end{equation}
Let $\pi$ be a uniformizer (generator of the unique maximal ideal) of $\mathcal{O}$, so that for any $c \in \mathcal{O}$, $c$ can be written $c = c_0 + c_1\pi + c_2\pi^2 + c_3\pi^3 + \ldots$, with $c_i \in \{0,1\}$.

Factoring out the highest power of $\pi$, the coefficient $c$ of any variable $x$ can be written in the form $c = \pi^r(c_0 + c_1\pi + c_2\pi^2 + c_3\pi^3 + \ldots)$, $c_0 = 1$.  Such a variable is said to be at \textit{level $r$}, and the value of $c_1$ is said to be its \textit{$\pi$-coefficient}.  By the simple change of variables $\pi^r x^6 = \pi^{r-6}(\pi x)^6 = \pi^{r-6}y^6$, we may rewrite (\ref{eq}) so that every variable has level $0 \le r \le 5$.

Because the proof of Knapp's conjecture involves inspecting numerous subcases, it is useful to introduce a compact notation to convey important information about particular forms and contractions performed on them.  I use the notation $(s_0, s_1, s_2, s_3, s_4, s_5)$ to indicate any form having \textit{at least} $s_i$ variables at level $i$.  For example, $(4, 2, 1, 0, 0, 0)$ indicates a form with at least four variables at level 0, at least two variables at level 1, at least one variable at level 2, and as few as no variables at any higher level.  As stated above, it should be assumed in this paper that the total number of variables always adds to 7, even when fewer than 7 are indicated by this notation.

To indicate that variables in a particular level have different $\pi$-coefficients, the number of variables in that level are partitioned into two numbers stacked vertically.  For example, $(^3_1, 2, 1, 0, 0, 0)$ indicates a form as above with four variables in level 0, three of which are in one $\pi$-coefficient class, and one of which is in the other.  Note that this notation does not give any indication of which $\pi$-coefficient class contains which number of variables.

By multiplying through by $\pi$, we may increase the level of each variable by 1, and then by a change of variables, change the level of all variables in level 6 to level 0.  To indicate the change that this produces, I use an arrow, with a \textit{c} above:

$$ (4,2,1,0,0,0) \xrightarrow{c} (0,4,2,1,0,0).$$

Applying this change multiple times (which will also be indicated by a  \textit{c}), we may effect a cyclic permutation of the levels of the variables (cf. Lemma 3 in \cite{knapp2016solubility}).  This is often useful, e.g, to move the level with a certain number of variables to level 0.  For example:

$$ (0,0,4,2,1,0) \xrightarrow{c} (4,2,1,0,0,0).$$

In order to construct a nontrivial solution to (\ref{eq}), we combine partial solutions using the operation of contraction.  Suppose we have a nontrivial solution modulo a power of $\pi$,  $$a_{t_1}\alpha_1^6 + \ldots + a_{t_k}\alpha_k^6 = \pi^j m, \pi \nmid m$$ for some subset of terms of (\ref{eq}) and choice of $\alpha_i$.  The change of variables $x_{t_i} = \alpha_i y$ gives $$a_{t_1}(\alpha_1 y)^6 + \ldots + a_{t_k}(\alpha_k y)^6 = (\pi^j m )y^6.$$  In this way, the variables $x_{t_1}, \ldots, x_{t_k}$ are replaced with a new variable $y$ at level $j.$  For the purposes of this proof, all contractions are of two variables in the same level.  Knapp outlines the four types of contraction which I use for this proof in his Lemmas 5 and 8 of \cite{knapp2016solubility}. I restate them verbatim here as Lemmas \ref{two} and \ref{three}, without proof, and they will hitherto be referenced by the numbering of this paper.

\begin{lemma}
\label{two}
Suppose that $x$ and $y$ are variables at level $k$.

If $x$ and $y$ have different $\pi$-coefficients, then they can be contracted to a variable $T$ at level $k + 1$. Moreover, we can arrange so that $T$ has whichever $\pi$-coefficient we like.

If $x$ and $y$ have the same $\pi$-coefficient, then they can be contracted to a variable $T$ at level $k + 2$.

Also, in this case they can be contracted to a variable $T$ at level at least $k + 3$.
  
We note that in the case where $x$ and $y$ have the same $\pi$-coefficient, we cannot control the $\pi$-coefficient of $T$. Moreover, if we contract to level at least $k + 3$, then we cannot control the exact level of $T$.
\end{lemma}

I indicate a contraction of two variables with different $\pi$-coefficients with an arrow and a \textit{d} above, e.g.,

$$(^3_1, 2, 1, 0, 0, 0) \xrightarrow{d} (2, 3, 1, 0, 0, 0).$$

I indicate a contraction of two variables with the same $\pi$-coefficients resulting in a variable exactly 2 levels higher with \textit{s2} above the arrow, e.g.,

$$(^3_1, 2, 1, 0, 0, 0) \xrightarrow{s2} (^1_1, 2, 2, 0, 0, 0).$$

I indicate a contraction of two variables with the same $\pi$-coefficients resulting in a variable at least 3 levels higher with \textit{s3} above the arrow, e.g.,

$$(^3_1, 2, 1, 0, 0, 0) \xrightarrow{s3} (^1_1, 2, 1, 1, 0, 0).$$

Note that in the case of an \textit{s3}-contraction, the example given is just one of the possible outcomes of the contractions, since the resulting variable may appear in any level at least 3 higher than the level of the contracted variables.

\begin{lemma}
\label{three}
Suppose that the form contains at least 3 variables at level $k$ which all have the same $\pi$-coefficient.  Suppose also that the coefficients belong to $\mathbb{Q}_2(\sqrt{-1})$ or $\mathbb{Q}_2(\sqrt{-5})$.  Then there are two variables at level $k$ which can be contracted to a variable at level at least $k + 4$.
\end{lemma}

I indicate such a contraction with \textit{t} above the arrow, e.g.,

$$(^3_1, 2, 1, 0, 0, 0) \xrightarrow{t} (^1_1, 2, 1, 0, 1, 0).$$

Contractions can be used to demonstrate the existence of a nontrivial solution by the following consequence of Hensel's Lemma specialized to exponent $d=6$, and fields $k=\mathbb{Q}_2(\sqrt{-1})$ and $k=\mathbb{Q}_2(\sqrt{-5})$ (cf. Lemma 4 of \cite{knapp2016solubility}).

\begin{lemma}[Hensel's Lemma]
Let $x_i$ be a variable of (\ref{eq}) at level $h$.  Suppose that $x_i$ can be used in a contraction of variables (or one in a series of contractions) which produces a new variable at level at least $h+5$.  Then (\ref{eq}) has a nontrivial solution.
\end{lemma}

The aim is then to show that a series of contractions can be performed which ``raises" a variable five levels.  To simplify showing that such a series of contractions can be formed, one may designate a level $k$.  Thus, by showing that contractions involving at least one variable from the designated level produce a variable at level $k+5$, a nontrivial solution follows.  To further simplify the counting of levels that a variable is raised, I borrow the notion of \textit{primary variable} from \cite{davenport1963homogeneous}.  A variable in the designated level is considered to be a primary variable, as is any variable formed from a contraction involving a primary variable.  The presence of a primary variable in level $k+5$ indicates the existence of a nontrivial solution.  Further, it is often convenient to apply a cyclic permutation to the level of the variables and then designate level 0.  In this way, we may show the existence of a nontrivial solution by raising a primary variable to level 5.  I indicate with an asterisk a level (or $\pi$-coefficient class within a level) which contains a primary variable.  For example:

$$(^3_1, 1, 1, 0, 1, 0) \xrightarrow{t} (^1_1, 1, 1, 0, 2*, 0).$$

Note that if any contraction produces a new variable 5 levels up, a nontrivial solution immediately follow by Hensel's Lemma.  Thus in the proofs below it is always assumed that an \textit{s3}-contraction produces a variable exactly 3 or 4 levels up, and that a \textit{t}-contraction produces a variable exactly 4 levels up.

\section{Proof of the conjecture}

\begin{lemma} \label{max7}
Suppose that (\ref{eq}) has at least seven variables at the same level. Then (\ref{eq}) has a nontrivial solution.
\end{lemma}
\begin{proof}
By the pigeonhole principle, there are three pairs of variables each sharing the same $\pi$-coefficient.  Perform three \textit{s2}-contractions on these three pairs to produce three new variables exactly 2 levels up.  Again, by the pigeonhole principle, two of these new variables have the same $\pi$-coefficient.  Perform an \textit{s3}-contraction on this pair of variables to produce a new variable at least 3 additional levels up.  A (nontrivial) solution follows from Hensel's Lemma.
\end{proof}

(Note that Lemma \ref{max7} is stated in greater generality than necessary, as we assume that every form considered in this proof has exactly seven variables.)

\begin{lemma}\label{donefrom4}
Designate level $k$.  If after a series of contractions there are at least two variables in level $k+4$, at least one of which is primary, then (\ref{eq}) has a nontrivial solution.
\end{lemma}
\begin{proof}
Any contraction involving the primary variable will result in a variable at level at least $k+5$.  The nontrivial solution follow from Hensel's lemma.
\end{proof}

\begin{lemma} \label{max5same}
Suppose that (\ref{eq}) has at least five variables at the same level with the same $\pi$-coefficient.  Then (\ref{eq}) has a nontrivial solution.
\end{lemma}
\begin{proof}
Without loss of generality, assume that the 5 variables appear in level 0.  Then,
$$(^{5}_{0},0,0,0,0,0) \xrightarrow{t}
(^{3}_{0},0,0,0,1*,0) \xrightarrow{t}
(^{1}_{0},0,0,0,2*,0).$$
The result follows from Lemma \ref{donefrom4}.
\end{proof}

\begin{lemma} \label{donefrom2}
Designate level $k$.  If after a series of contractions there are at least two variables in the same level with the same $\pi$-coefficient, at least level $k+2$, at least one of which is primary, then (\ref{eq}) has a nontrivial solution.
\end{lemma}
\begin{proof}
Perform an $s3$-contraction on the two variables.  The solution follows from Hensel's lemma.
\end{proof}

\begin{lemma} \label{donefrom2with3}
Designate level $k$.  If after a series of contractions there are at least three variables in the same level, at least level $k+2$, at least two of which are primary, then (\ref{eq}) has a nontrivial solution.
\end{lemma}
\begin{proof}
This follows from Lemma \ref{donefrom2} and the pigeonhole principle.
\end{proof}

\begin{lemma}\label{max6}
Suppose that (\ref{eq}) has at least six variables at the same level.  Then (\ref{eq}) has a nontrivial solution.
\end{lemma}
\begin{proof}
Without loss of generality, assume that the six variables appear in level 0.  By lemma \ref{max5same}, assume there are at most four variables with the same $\pi$-coefficient in level 0.  First suppose there are four variables in one of the $\pi$-coefficient classes in level 0.  Then, after the series of contractions
$$(^{4}_{2},0,0,0,0,0) \xrightarrow{s2}
(^{2}_{2},0,1*,0,0,0) \xrightarrow{s2}
(^{0}_{2},0,2*,0,0,0) \xrightarrow{s2}
(^{0}_{0},0,3*,0,0,0),$$
the result follows from Lemma \ref{donefrom2with3}.  Next assume there are three variables in each $\pi$-coefficient class.  After the series of contractions 

$$(^{3}_{3},0,0,0,0,0) \xrightarrow{t}
(^{1}_{3},0,0,0,1*,0) \xrightarrow{t}
(^{1}_{1},0,0,0,2*,0),$$
the result follows from Lemma \ref{donefrom4}.
\end{proof}

\begin{lemma} \label{doneaftert}
Suppose that (\ref{eq}) has three variables with the same $\pi$-coefficient in level $k$ and at least one variable in level $k+4$.  Then (\ref{eq}) has a nontrivial solution.
\end{lemma}
\begin{proof}
Perform a \textit{t}-contraction among the three variables.  The solution follows from Lemma \ref{donefrom4}.
\end{proof}

\begin{lemma} \label{doneafterd}
Suppose that (\ref{eq}) has two variables with different $\pi$-coefficient in level $k$ and at least one variable in each of levels $k+1$ and $k+2$.  Then (\ref{eq}) has a nontrivial solution.
\end{lemma}
\begin{proof}

After the series of contractions: $$(^{1}_{1},1,1,0,0,0) \xrightarrow{d}
(0,{}^{1*}_{1},1,0,0,0) \xrightarrow{d}
(0,0,{}^{2*}_{0},0,0,0),$$
the solution follows from Lemma \ref{donefrom2}.
\end{proof}

\begin{lemma} \label{doneafterds2}
Suppose that (\ref{eq}) has at least two variables with the same $\pi$-coefficient in level $k$ and at least two variables with differing $\pi$-coefficients in level $k+1$.  Then (\ref{eq}) has a nontrivial solution.
\end{lemma}
\begin{proof}
After the series of contractions:
$$(^{2}_{0},{}^{1}_{1},0,0,0,0) \xrightarrow{s2}
(0,{}^{1}_{1},1*,0,0,0) \xrightarrow{d}
(0,0,{}^{2*}_{0},0,0,0),$$
the solution follows from Lemma \ref{donefrom2}.
\end{proof}

\begin{lemma} \label{slide}
Suppose that (\ref{eq}) has two variables in level $k$, and at least one variable in levels $k+1$, $k+2$, ... $k+t-1$, then contractions can be performed to produce a variable at level at least $k+t$.
\end{lemma}
\begin{proof}
Any two variables in the same level can be contracted to a variable at least one level higher.  By repeated contractions, we obtain the desired variable.
\end{proof}

\begin{lemma} \label{doneafters2}
Suppose that (\ref{eq}) has two variables with the same $\pi$-coefficient in level $k$ and at least one variable in each of levels $k+2$ and $k+3$.  Then (\ref{eq}) has a nontrivial solution.
\end{lemma}
\begin{proof}
Perform an \textit{s2}-contraction on the two variables.  By Lemma \ref{donefrom2}, assume the resulting variable and the existing variable in level $k+2$ have different $\pi$-coefficients.  Perform a \textit{d}-contraction to create a variable in level $k+3$ with the same $\pi$-coefficient as the existing variable.  The solution follows from Lemma \ref{donefrom2}.
\end{proof}

\begin{lemma} \label{doneafters2_2}
Suppose that (\ref{eq}) has two variables with the same $\pi$-coefficient in level $k$ and at least two variables in level $k+2$ not having the same $\pi$-coefficient.  Then (\ref{eq}) has a nontrivial solution.
\end{lemma}
\begin{proof}
Perform an \textit{s2}-contraction on the two variables in level $k$.  The solution following from Lemma \ref{donefrom2}.
\end{proof}

\begin{lemma} \label{doneafters3_2}
Suppose that (\ref{eq}) has two variables with the same $\pi$-coefficient in level $k$ and at least two variables in level $k+3$ not having the same $\pi$-coefficient.  Then (\ref{eq}) has a nontrivial solution.
\end{lemma}
\begin{proof}
Perform an \textit{s3}-contraction on the two variables in level $k$.  If the new variable goes to level $k+3$, a solution follows from Lemma \ref{donefrom2}.  If the new variable goes to level $k+4$, perform a \textit{d}-contraction on the variables in level $k+3$ to produce a variable in level $k+4$.  A solution follows from any contraction on the two variables in level $k+4$.
\end{proof}

\begin{lemma} \label{doneafters3}
Suppose that (\ref{eq}) has two variables with the same $\pi$-coefficient in level $k$ and at least one variable in each of levels $k+3$ and $k+4$.  Then (\ref{eq}) has a nontrivial solution.
\end{lemma}
\begin{proof}
Perform an \textit{s3}-contraction on the two variables.
By Hensel's Lemma assume that the new variable is created at level $k+3$ or $k+4$.
By Lemma \ref{donefrom4}, assume it is level $k+3$.
The solution follows from Lemma \ref{slide} with $t=2$ and $k$ replaced by $k+3$.
\end{proof}

\begin{lemma} \label{empty234}
Suppose that (\ref{eq}) has at least four variables in level $k$ which can be used to form two pairs each with the same $\pi$-coefficient, and at least one variable in one of levels $k+2$, $k+3$, or $k+4$.  Then (\ref{eq}) has a nontrivial solution.
\end{lemma}
\begin{proof}
If level $k+2$ contains at least one variable, perform \textit{s2}-contractions on the two pairs.  The solution follows from Lemma \ref{donefrom2}.

If level $k+3$ contains at least one variable, perform \textit{s3}-contractions on the two pairs.  By Lemma \ref{donefrom4}, assume that both of the resulting variables don't go to level $k+4$, and by Lemma \ref{donefrom2} and the pigeonhole principle, assume that both variables don't go to level $k+3$.  Then the solution follows from Lemma \ref{slide} and Hensel's Lemma.

If level $k+4$ contains at least one variable, perform \textit{s3}-contractions on the two pairs.  By Lemma \ref{slide} and Hensel's Lemma, assume both variables don't go to level $k+3$, and thus at least one variable goes to level $k+4$.  The solution follows from Lemma \ref{donefrom4}.
\end{proof}

\begin{lemma} \label{oneofempty14}
Suppose that (\ref{eq}) has two variables in level $k$ with different $\pi$-coefficients, and has at least one variable in both of levels $k+1$ and $k+4$.  Then (\ref{eq}) has a nontrivial solution.
\end{lemma}
\begin{proof}
Without loss of generality, assume $k=0$.  Consider the series of contractions:

$$({}^{1}_{1},1,0,0,1,0) \xrightarrow{d}
(0,2*,0,0,1,0) \xrightarrow{s3}
(0,0,0,0,2*,0).$$
Here, because the s3-contraction involves a primary variable that has already been raised one level, if it creates a variable at least four levels up, then a solution follows from Hensel's Lemma.  Hence, we assume that it produces a variables exactly three levels up.  Perform a contraction with the other variable in level 4, and a solution follows from Hensel's Lemma.
\end{proof}

\begin{lemma} \label{empty14}
Suppose that (\ref{eq}) has at least four variables in level $k$ which can be used to form two pairs, one with the same $\pi$-coefficient and one with different $\pi$-coefficients, and has at least one variable in one of levels $k+1$ or $k+4$.  Then (\ref{eq}) has a nontrivial solution.
\end{lemma}
\begin{proof}
Suppose that level $k+1$ contains at least one variable.  After the contraction:

$$({}^{3}_{1},1,0,0,0,0) \xrightarrow{s2}
({}^{1}_{1},1,1,0,0,0)$$
the solution follows from Lemma \ref{doneafterd}.

Suppose that level $k+4$ contains at least one variable.  The solution follows from Lemma \ref{doneaftert}.
\end{proof}

\begin{lemma} \label{max5}
Suppose that (\ref{eq}) has at least five variables at the same level.  Then (\ref{eq}) has a nontrivial solution.
\end{lemma}
\begin{proof}
Without loss of generality, assume the five variables appear in level 0.  By lemma \ref{max6}, assume level 0 contains exactly five variables.  We will need to consider the locations of the remaining two variables.

By lemma \ref{max5same}, there are at most four variables with the same $\pi$-coefficient in level 0.  By the pigeonhole principle, one may construct two pairs of variables, one with the same $\pi$-coefficient and one with different $\pi$-coefficients.  Thus by Lemma \ref{empty14}, assume there are no variables in levels 1 or 4.

By the pigeonhole principle, one may also construct two pairs of variables, each with the same $\pi$-coefficient.  Thus by Lemma \ref{empty234}, assume that there are no variables in levels 2, 3, or 4.

Therefore, assume the remaining two variables are in level 5.  They can be contracted to a variable at least 1 level higher.  By Hensel's Lemma, assume this resulting variable is not at least 5 levels higher, and so is formed in level 6, 7, 8, or 9.  By a change of variables, the variable can be moved down six levels to level 0, 1, 2, or 3.  The resulting form is covered by one of the preceding cases in this Lemma or by Lemma \ref{max6}, and a solution follows.
\end{proof}

Note that the change of variables used in the proof of the preceding lemma allows a variable in level $k$ to be regarded as a variable in any level $\ell \equiv k \pmod{6}$.  For the remainder of the article, we will exploit this fact implicitly, referring only to the lemma which determines the chosen level $\ell$.

\begin{lemma}\label{3and2same}
Suppose that (\ref{eq}) has at least three variables in level $k$ and at least two variables in level $k+5$.  Suppose further that it is not the case that all of the variables in level $k$ have the same $\pi$-coefficient and all the variables in level $k+5$ have the same $\pi$-coefficient. Then (\ref{eq}) has a nontrivial solution.
\end{lemma}
\begin{proof}

Without loss of generality, assume $k=0$.  By the pigeonhole principle, at least two variables in level 0 have the same $\pi$-coefficient.

Suppose two of the variables in level 5 have differing $\pi$-coefficients. First consider the case where not all of the variables in level 0 have the same $\pi$-coefficient.  After the change of variables and series of contractions:
$$({}^{2}_{1},0,0,0,0,{}^{1}_{1}) \xrightarrow{c} ({}^{1}_{1},{}^{2}_{1},0,0,0,0) \xrightarrow{d} (0,{}^{2*}_{2},0,0,0,0) \xrightarrow{d,d} (0,0,{}^{2*}_{0},0,0,0),$$
the solution follows from Lemma \ref{donefrom2}.

Now, consider the case where all of the variables in level 0 have the same $\pi$-coefficient.  We examine the locations of the remaining two variables.  First, assume level 3 is occupied.  After the change of variables:
$$({}^{3}_{0},0,0,1,0,{}^{1}_{1}) \xrightarrow{c} ({}^{1}_{1},{}^{3}_{0},0,0,1,0),$$
a solution follows from Lemma \ref{oneofempty14}.  Thus, assume level 3 is unoccupied.  By Lemma \ref{doneaftert}, assume that level 4 is unoccupied, and by Lemma \ref{doneafterd} with $k=5$, assume level 1 is unoccupied.  By Lemma \ref{doneafters2_2} with $k=0$, assume all variables in level 2 have the same $\pi$-coefficient.  If level 2 contains two variables, a solution follows from Lemma \ref{doneafters3_2} with $k=2$.  If level 0 contains four variables, a \textit{d}-contraction can be performed on level 5 so that there are five variables in level 0 with the same $\pi$-coefficient, and a solution follows from Lemma \ref{max5same}.  If the variables in level 5 can be used to form two pairs with differing $\pi$-coefficients, then two \textit{d}-contractions can be performed, and again a solution follows from Lemma \ref{max5same}.  If the variables in level 5 can be used to form two pairs, one with the same $\pi$-coefficients and one with differing $\pi$-coefficients, then a solution follows from Lemma  \ref{empty14} with $k=5$.  One case remains:

$$({}^{3}_{0},0,1,0,0,{}^{2}_{1}) \xrightarrow{c} ({}^{2}_{1},{}^{3}_{0},0,1,0,0) \xrightarrow{d} (1,{}^{4}_{0},0,1,0,0) \xrightarrow{s3}(1,{}^{2}_{0},0,1,1,0).$$
A solution follows from Lemma \ref{doneafters2} with $k=1$.  (Here, the \textit{s3}-contraction is performed with a variable which has already been raised one level, and so by Hensel's Lemma, we may assume that the resulting variable is created exactly three levels up.)

Now, suppose that the two variables in level 5 have the same $\pi$-coefficient.
The hypothesis implies that the three variables in level 0 do not have the same $\pi$-coefficient.  After the change of variables and series of contractions:
$$({}^{2}_{1},0,0,0,0,{}^{2}_{0}) \xrightarrow{c} ({}^{2}_{0},{}^{2}_{1},0,0,0,0) \xrightarrow{s2} (0,{}^{2}_{1},1*,0,0,0) \xrightarrow{d} (0,1,{}^{2*}_{0},0,0,0),$$
the solution follows from Lemma \ref{donefrom2}.

\end{proof}

\begin{lemma}\label{max4}
Suppose that (\ref{eq}) has at least four variables at the same level.  Then (\ref{eq}) has a nontrivial solution.
\end{lemma}
\begin{proof} By Lemma \ref{max5}, we may assume that the level with at least four variables has exactly four variables; without loss of generality, assume that this is level 0.

First, suppose at least two of the remaining three variables are in level 5.  By Lemma \ref{3and2same}, assume that all the variables in level 0 have the same $\pi$-coefficient, and that all the variables in level 5 have the same $\pi$-coefficient.  Thus the four variables can be used to form two pairs, each having the same $\pi$-coefficient.  Perform an \textit{s3}-contraction on the two variables in level 5.  After a change of variables, the resulting variable will be in level 2 or 3, and the solution follows from Lemma \ref{empty234}.  Thus, assume that level 5 contains at most one variable.  We will examine the locations of the remaining two variables.

By Lemma \ref{empty234} if the four variables form two pairs with matching $\pi$-coefficients we may assume levels 2, 3, and 4 are unoccupied, and therefore that level 1 contains at least two variables.  These may be contracted to a variable in level 2 or 3, and so the solution follows again from Lemma \ref{empty234}.

Thus, assume that three of the four variables in level 0 are in one $\pi$-coefficient class, and the one remaining variable is in the other class.  By Lemma \ref{empty14}, we may assume levels 1 and 4 are unoccupied, and by Lemma \ref{doneafters2} that one of levels 2 or 3 is unoccupied.

If two of the remaining variables lie in level 2 and have the same $\pi$-coefficient, they can be contracted to a variable in level 4 and the solution follows from Lemma \ref{doneaftert}.  Suppose they have differing $\pi$-coefficients.  After the contraction: 
$$(^{3}_{1},0,{}^{1}_{1},0,0,0) \xrightarrow{s2}
({}^{1}_{1},0,{}^{2*}_{1},0,0,0),$$
the solution follows from Lemma \ref{donefrom2}.

Thus we may assume both variables lie in level 3.  If they have differing $\pi$-coefficients, by the pigeonhole principle there is a pair of variables in level 0 with the same $\pi$-coefficients, and a solution follows from Lemma \ref{doneafters3_2}.  Thus assume they have the same $\pi$-coefficient.  If there is a variable in level 5, by a change of variables we may consider the variables in level 0 to be in level 6, and the solution follows from Lemma \ref{doneafters2} with $k=3$.  Thus we may assume there are three variables in level 3, all with the same $\pi$-coefficient.  After the change of variables and series of contractions:
$$({}^{3}_{1},0,0,{}^{3}_{0},0,0) \xrightarrow{c} ({}^{3}_{0},0,0,{}^{3}_{1},0,0) \xrightarrow{t} (1,0,0,{}^{3}_{1},1*,0) \xrightarrow{d} (1,0,0,2,2*,0),$$
the solution follows from Lemma \ref{donefrom4}.
\end{proof}

\begin{lemma} \label{3diffthen2}
Suppose that (\ref{eq}) has at least three variables in level $k$ not all having the same $\pi$-coefficient, and at least two variables in level $k+1$.  Then (\ref{eq}) has a nontrivial solution.
\end{lemma}
\begin{proof}
Without loss of generality, assume $k=0$.  By Lemma \ref{max4}, level 0 has exactly three variables.  If there are at least three variables in level 1, by multiplying by $\pi^5$ we may consider level 5 and 6 to have three variables each.  By a change of variables, the variables in level 6 may be considered to be in level 0.  Then a solution follows from Lemma \ref{3and2same}.  Thus assume level 1 contains exactly two variables. Suppose that the two variables in level 1 have different $\pi$-coefficients.  After the series of contractions:
$$({}^{2}_{1},{}^{1}_{1},0,0,0,0) \xrightarrow{s2} (1,{}^{1}_{1},1*,0,0,0) \xrightarrow{d} (1,0,{}^{2*}_{0},0,0,0),$$
the solution follows from Lemma \ref{donefrom2}.

Thus, suppose that the two variables in level 1 have the same $\pi$-coefficient.  By Lemma \ref{doneafterd}, assume level 2 is unoccupied.  By Lemma \ref{oneofempty14}, assume level 4 is unoccupied.  Suppose level 5 is occupied.  After the contractions and change of variables:
$$({}^{2}_{1},{}^{2}_{0},0,0,0,1) \xrightarrow{d} (1,{}^{3}_{0},0,0,0,1) \xrightarrow{c}
({}^{3}_{0},0,0,0,1,1) \xrightarrow{t}
(1,0,0,0,2*,1),$$
the solution follows from Lemma \ref{donefrom4}.  Thus, assume level 5 is unoccupied, and the remaining two variables lie in level 3.  If they have differing $\pi$-coefficients, then they can be contracted to a variable in level 4, and again a solution follows from Lemma \ref{oneofempty14}.  If they have the same $\pi$-coefficient, they can be contracted to a variable in level 5, and a solution follows from the case described above.
\end{proof}

\begin{lemma}\label{max3_1}
Suppose that (\ref{eq}) has at least three variables at the same level, not all having the same $\pi$-coefficient.  Then (\ref{eq}) has a nontrivial solution.
\end{lemma}
\begin{proof}

By Lemma \ref{max4}, we may assume that the level(s) with at least three variables have exactly three variables.  Without loss of generality assume that one of the levels with three variables is level 0.  By Lemma \ref{3and2same}, we may assume there is at most one variable in level 5.  We examine the locations of the remaining three variables.

By Lemma \ref{doneafters2}, we may assume that at least one of levels 2 and 3 is unoccupied, and by Lemma \ref{doneafters3} that at least one of levels 3 and 4 is unoccupied.  By Lemma \ref{doneafterd}, we may assume that at least one of levels 1 and 2 is unoccupied, and by Lemma \ref{oneofempty14}, we may assume at least one of levels 1 and 4 is unoccupied.

Suppose level 1 is occupied.  By Lemma \ref{3diffthen2}, assume level 1 contains exactly one variable.  By the above considerations, assume levels 2 and 4 are unoccupied.  If there are two variables with different $\pi$-coefficients in level 3, they can be contracted to a variable in level 4 and the solution follows from Lemma \ref{oneofempty14}.  Thus, assume there are at least two variables in level 3, all having the same $\pi$-coefficient.  By Lemma \ref{doneafters2} with $k=3$, assume that level 5 is unoccupied, and thus that level 3 has exactly three variables with the same $\pi$-coefficient, and then by Lemma \ref{doneaftert} with $k=3$, there is a nontrivial solution.  Thus assume level 1 is unoccupied.

Suppose level 2 has at least two variables, and so assume levels 1 and 3 are unoccupied.  By Lemma \ref{doneafters2_2} with $k=0$, assume the variables in level 2 have the same $\pi$-coefficient.  By Lemma \ref{doneaftert} with $k=2$, assume there are at most two variables in level 2, and thus level 4 must contain at least one variable.  If level 5 also contains a variable, then the solution follows from Lemma \ref{doneafters2} with $k=2$, and so assume level 5 is unoccupied, and thus level 4 contains two variables.  If the two variables have differing $\pi$-coefficients, the solution follows from Lemma \ref{doneafters2_2} with $k=2$.  If the two variables have the same $\pi$-coefficient, the solution follows from Lemma \ref{doneafters2_2} with $k=4$.  Thus assume that level 2 contains at most one variable.

Suppose level 3 is occupied.  By the above, assume levels 1, 2, and 4 are unoccupied, and so level 3 contains exactly three variables and level 5 contains exactly one.  By the pigeonhole principle, at least two of the variables in level 3 must have the same $\pi$-coefficient, and the solution follows from Lemma \ref{doneafters3_2} with $k=3$.  Thus assume that level 3 is unoccupied.

Suppose level 4 is occupied.  By the preceding reasoning, assume levels 1 and 3 are unoccupied and levels 2 and 5 contain at most one variable.  Thus assume level 4 contains at least two variables.  If two of these variables have the same $\pi$-coefficient, then the solution follows from Lemma \ref{doneafters2_2} with $k=4$.  Thus assume level 4 contains exactly two variables with differing $\pi$-coefficients, level 2 contains one variable, and level 5 contains one variable.  The solution follows from Lemma \ref{doneafterd} with $k=4$.
\end{proof}

\begin{lemma}\label{33}
Suppose that (\ref{eq}) has at least three variables at level $k$ and at least three variables at level $k+1$.  Then (\ref{eq}) has a nontrivial solution.
\end{lemma}
\begin{proof}
Without loss of generality, assume $k=0$.  By Lemma \ref{max4} and Lemma \ref{max3_1}, assume that levels 0 and 1 each have exactly three variables with the same $\pi$-coefficient.  By Lemma \ref{doneaftert}, assume that levels 4 and 5 are unoccupied.  Thus, assume the remaining variable is in level 2 or 3.  First, suppose it is in level 2.  After the contraction

$$({}^{3}_{0},{}^{3}_{0},1,0,0,0) \xrightarrow{s2} ({}^{3}_{0},1,1,1,0,0)$$
the solution follows from Lemma \ref{doneafters2}.  Now, suppose the variable is in level 3.  After an initial \textit{s2}-contraction, there are two possible cases:

$({}^{3}_{0},{}^{3}_{0},0,1,0,0) \xrightarrow{s2}$

\begin{itemize}
    \item $({}^{3}_{0},1,0,{}^{1}_{1},0,0) \xrightarrow{d} ({}^{3}_{0},1,0,0,1,0),$
    and the solution follows from Lemma \ref{doneaftert}.
    \item $({}^{3}_{0},1,0,{}^{2*}_{0},0,0),$
    and the solution follows from Lemma \ref{donefrom2} with $k=1$.
\end{itemize}

\end{proof}

\begin{lemma}\label{max3}
Suppose that (\ref{eq}) has at least three variables at the same level.  Then (\ref{eq}) has a nontrivial solution.
\end{lemma}
\begin{proof}
By Lemma \ref{max4}, we may assume that any level with at least three variables has exactly three variables.  By Lemma \ref{max3_1}, we may assume that the variables in any level with three variables all have the same $\pi$-coefficient.   Without loss of generality assume that one of the levels with three variables is level 0.

By Lemma \ref{doneafters2}, we may assume that at least one of levels 2 and 3 is unoccupied.  By Lemma \ref{doneaftert}, we may assume that level 4 is unoccupied. By Lemma \ref{33}, assume that levels 5 and 1 each contain at most two variables.

By Lemma \ref{doneafterds2}, we may assume that all variables in level 1 have the same $\pi$-coefficient.  By Lemma \ref{doneafters2_2},  assume that all variables in level 2 have the same $\pi$-coefficient.  If there are at least two variables in level 2, they may be contracted to a variable in level 4, and a solution follows from Lemma \ref{doneaftert}.  Thus assume that level 2 contains at most one variable.  If there are two variables in level 3 having differing $\pi$-coefficients, they may be contracted to a variable in level 4 and a solution follows from Lemma \ref{doneaftert}, and so assume all variables in level 3 have the same $\pi$-coefficient.  By Lemma \ref{3and2same}, we may assume that all variables in level 5 have the same $\pi$-coefficient.

First suppose that level 5 contains no variables, and so there are four variables that occupy levels 1, 2, and 3.  By the assumptions made, assume level 1 contains at most two, level 2 contains at most one, and thus level 3 contains a variable.  Therefore, assume that level 2 is unoccupied, and so level 1 is occupied, and level 3 has at least two variables.  A solution follows from Lemma \ref{doneafters3} with $k=3$.  Thus we may assume that level 5 contains either one or two variables, and so (by Lemma \ref{doneafters2} with $k=3$ and the assumptions above) that levels 2 and 3 together contain at most one variable, and thus that level 1 contains either one or two variables.

Now suppose that level 5 contains exactly one variable.  By the above assumptions, level 1 contains exactly two variables, and levels 2 and 3 together contain exactly one variable.  Perform an \textit{s2}-contraction on the two variables in level 1 to create a new variable in level 3.  If level 2 contains a variable, then a solution follows from Lemma \ref{doneafters2} with $k=0$.  Thus suppose that level 2 is unoccupied and level 3 contains exactly one variable.  If this variable has the same $\pi$-coefficient as the variable resulting from the contraction, the solution follows from Lemma \ref{donefrom2}.  However, if the $\pi$-coefficients of the two variables differ, they can be contracted to a variable in level 4 and the solution follows from Lemma \ref{doneaftert}.

Thus, we may assume that level 5 contains exactly two variables having the same $\pi$-coefficient and level 1 contains at least one variable.  Perform an \textit{s2}-contraction on a pair of variables from level 0 to produce a variable in level 2.  A solution follows from Lemma \ref{doneafters2} with $k=5$.
\end{proof}

\begin{lemma}\label{max2}
Suppose that (\ref{eq}) has a least two variables at the same level.  Then (\ref{eq}) has a nontrivial solution.
\end{lemma}
\begin{proof}

By Lemma \ref{max3}, we may assume that any level with at least two variables has exactly two variables.  It is possible by applying an appropriate cyclic permutation to have a form with at least $n+1$ variables in the first $n$ levels for $1 \le n \le 6$  (cf. Lemma 3 of \cite{knapp2016solubility}).  This process is known as normalization.  It follows that we may assume that level 0 has two variables, level 1 is occupied, and level 5 has at most one variable.

First suppose that the two variables in level 0 have differing $\pi$-coefficients.  By Lemma \ref{doneafterd}, assume level 2 is unoccupied.  Thus levels 0 through 2 contain at most four variables, and so by normalization assume that level 3 is occupied.  Perform a \textit{d}-contraction on the pair in level 0 so that there are two variables in level 1 with the same $\pi$-coefficient, and use this pair to perform an \textit{s3}-contraction from level 1.  By Hensel's Lemma, assume this results in a new variable in level 4.  By Lemma \ref{donefrom4}, assume level 4 is unoccupied, and therefore by normalization that level 3 contains two variables.  If the two variables in level 3 have differing $\pi$-coefficients, a solution follows from Lemma \ref{doneafters2_2} with $k=1$.  If these two variables have the same $\pi$-coefficient, omit the initial \textit{d}-contraction, and a solution follows from Lemma \ref{doneafters2} with $k=3$.

Thus, we may assume that the two variables in level 0 have the same $\pi$-coefficient.  Suppose that level 1 contains just one variable.  By normalization, assume level 2 is occupied, and so by Lemma \ref{doneafters2} that level 3 is unoccupied, and again by normalization, that level 2 contains two variables, and by Lemma \ref{doneafters2_2} that they have the same $\pi$-coefficient.  By Lemma \ref{doneafters3} with $k=2$, assume that level 5 is unoccupied, and so level 4 contains the remaining two variables.  By Lemma \ref{doneafters2_2} with $k=2$, assume that they have the same $\pi$-coefficient.  A solution follows from Lemma \ref{doneafters3} with $k=4$.

Thus we may assume level 1 contains two variables, and by Lemma \ref{doneafterds2} that they have the same $\pi$-coefficient.  If level 2 is occupied, perform an \textit{s2}-contraction on the two variables in level 1, and a solution follows from Lemma \ref{doneafters2}.  Therefore, assume that level 2 is unoccupied and by normalization that level 3 is occupied.  By Lemma \ref{doneafters3}, assume level 4 is unoccupied, and so by normalization that level 3 contains two variables, and by Lemma \ref{doneafters2_2} with $k=1$ that they have the same $\pi$-coefficient.  The remaining variable is in level 5, and a solution follows from Lemma \ref{doneafters2} with $k=3$.
\end{proof}

The forms we consider have seven variables and six levels, and thus there is some level that contains at least two variables.  By Lemma \ref{max2}, this completes the proof of the theorem.

\bibliographystyle{unsrt}
\bibliography{biblio}
\end{document}